\documentclass[reqno,a4paper,draft]{amsart}

\usepackage{enumitem}
\setenumerate{label=\textnormal{(\arabic*)}}

\usepackage{amsmath,amssymb,dsfont,verbatim,bm,array,mathtools}
\usepackage[latin1]{inputenc}
\usepackage{booktabs}
\usepackage[raggedright]{titlesec}
\usepackage{mathtools}

\titleformat{\chapter}[display]
{\normalfont\huge\bfseries}{\chaptertitlename\\thechapter}{20pt}{\Huge}
\titleformat{\section}
{\normalfont\Large\bfseries\center}{\thesection}{1em}{}
\titleformat{\subsection}
{\normalfont\large\bfseries}{\thesubsection}{1em}{}
\titleformat{\subsubsection}[runin]
{\normalfont\normalsize\bfseries}{\thesubsubsection}{1em}{}
\titleformat{\paragraph}[runin]
{\normalfont\normalsize\bfseries}{\theparagraph}{1em}{}
\titleformat{\subparagraph}[runin]
{\normalfont\normalsize\bfseries}{\thesubparagraph}{1em}{}
\titlespacing*{\chapter} {0pt}{50pt}{40pt}
\titlespacing*{\section} {0pt}{3.5ex plus 1ex minus .2ex}{2.3ex plus .2ex}
\titlespacing*{\subsection} {0pt}{3.25ex plus 1ex minus .2ex}{1.5ex plus .2ex}
\titlespacing*{\subsubsection}{0pt}{3.25ex plus 1ex minus .2ex}{1.5ex plus .2ex}
\titlespacing*{\paragraph} {0pt}{3.25ex plus 1ex minus .2ex}{1em}
\titlespacing*{\subparagraph} {\parindent}{3.25ex plus 1ex minus .2ex}{1em}

\input xypic
\xyoption{all}

\newtheorem{theorem}{Theorem}[section]
\newtheorem{lemma}[theorem]{Lemma}
\newtheorem{proposition}[theorem]{Proposition}

\newtheorem{conjecture}[theorem]{Conjecture}

\theoremstyle{definition}

\newtheorem{examples}[theorem]{Examples}

\theoremstyle{remark}
\newtheorem{remark}[theorem]{Remark}

\DeclareMathOperator{\Cent}{Cent}

\DeclareMathOperator{\Jac}{Jac}

\begin{document}
\title{The two-dimensional Centralizer Conjecture}

\author{Vered Moskowicz}
\address{Department of Mathematics, Bar-Ilan University, Ramat-Gan 52900, Israel.}
\email{vered.moskowicz@gmail.com}


\begin{abstract}
A result by C. C.-A. Cheng, J. H. Mckay and S. S.-S. Wang says the following:
Suppose $\Jac(A,B) \in \mathbb{C}^*$ and $\Jac(A,w)=0$ for $A,B,w \in \mathbb{C}[x,y]$.
Then $w \in \mathbb{C}[A]$.

We show that in CMW's result it is possible to replace $\mathbb{C}$ by any field of characteristic zero,
and we conjecture the following 'two-dimensional Centralizer Conjecture over $D$':
Suppose $\Jac(A,B) \in D^*$ and $\Jac(A,w)=0$ for $A,B,w \in D[x,y]$,
$D$ is an integral domain of characteristic zero.
Then $w \in D[A]$.

We show that if the famous two-dimensional Jacobian Conjecture is true, 
then the two-dimensional Centralizer Conjecture is true.
\end{abstract}

\maketitle

\section{Introduction}
Throughout this note, $k$ denotes a field of characteristic zero,
$D$ denotes an integral domain of characteristic zero
($D$ is commutative) with field of fractions $Q(D)$.
The group of invertible elements of a ring $R$ will be denoted by $R^*$.

In ~\cite[Theorem 1]{wang younger}, C. C.-A. Cheng, J. H. Mckay and S. S.-S. Wang have shown that
`the centralizer with respect to the Jacobian' of an element $A \in \mathbb{C}[x,y]$ 
which has a Jacobian mate, equals the polynomial ring in $A$ over $\mathbb{C}$:
Suppose $\Jac(A,B) \in \mathbb{C}^*$ and $\Jac(A,w)=0$ for $A,B,w \in \mathbb{C}[x,y]$.
Then $w \in \mathbb{C}[A]$.

We show that $\mathbb{C}$ can be replaced by $k$.

Then we conjecture that $k$ can be replaced by any integral domain of characteristic zero,
and call this 'the two-dimensional Centralizer Conjecture over $D$':
Suppose $\Jac(A,B) \in D^*$ and $\Jac(A,w)=0$ for $A,B,w \in D[x,y]$.
Then $w \in D[A]$. 

The famous two-dimensional Jacobian Conjecture says that every $k$-algebra
endomorphism $f: (x,y) \mapsto (p,q)$ of $k[x,y]$ 
having an invertible Jacobian,
$\Jac(p,q):=p_xq_y-p_yq_x \in k[x,y]^*=k^*$, is an automorphism
of $k[x,y]$; it was raised by O. H. Keller ~\cite{keller}, 
actually over $\mathbb{Z}$ not over $k$.

We show that if the two-dimensional Jacobian Conjecture is true  
then the two-dimensional Centralizer Conjecture is true.
Therefore, if the two-dimensional Centralizer Conjecture is false for some $D$, 
then the two-dimensional Jacobian Conjecture is false. 
We 'believe' that the two-dimensional Jacobian Conjecture is true,
so there should be no counterexample to the two-dimensional Centralizer Conjecture.


\section{CMW's theorem over $k$} 
We will show in two ways that in CMW's theorem it is possible to
replace $\mathbb{C}$ by $k$.
\subsection{First way: Independent of the original CMW's theorem}
We will apply the following observation of E. Formanek. 

\begin{proposition}[Formanek's observation]\label{lemma formanek} 
Let $F$ be any field.
Assume that $C,D \in F[x,y]$ are algebraically dependent over $F$.
Then there exist $h \in F[x,y]$
and $u(t), v(t) \in F[t]$,
such that
$C=u(h)$ and $D=v(h)$.
\end{proposition}

\begin{proof}
The proof of Formanek can be found in ~\cite{formanek}.
We wish to mention that it uses L\"uroth's theorem, a sharpening of E. Noether (characteristic zero)
and A. Schinzel (arbitrary field);
those three results can be found in ~\cite[first chapter]{schinzel}. 
\end{proof}

The following is a known result due to Jacobi (1841); 
we will only need the second statement of it.           
\begin{theorem}[Jacobi's theorem]\label{jacobi}
Let $F$ be any field, $C,D \in F[x,y]$.
\begin{itemize}
\item [(1)] If $C$ and $D$ are algebraically dependent over $F$,
then $\Jac(C,D)=0$.
\item [(2)] Assume that $F$ is of characteristic zero.
If $\Jac(C,D)=0$,
then $C$ and $D$ are algebraically dependent over $F$.
\end{itemize}
\end{theorem}

\begin{proof}
See ~\cite[page 8]{alg dep} or ~\cite[pages 19-20]{makar deri}.
\end{proof} 

Actually, in $(2)$ of Jacobi's theorem, the base field $F$ can be either a field of characteristic zero
or a field of large enough characteristic, as is shortly explained in ~\cite[page 8]{alg dep}.
However, we are only interested in the zero characteristic case,
since we are interested in the (two-dimensional) Jacobian Conjecture,
which has a counterexample over a field of prime characteristic $P$,
for example, $(x,y) \mapsto (x+x^P,y)$.

Now we are ready to prove CMW's result over $k$:
 
\begin{theorem}[CMW's theorem over $k$]\label{cmw char zero}
Suppose $\Jac(A,B) \in k^*$ and $\Jac(A,w)=0$
for $A,B,w \in k[x,y]$.
Then $w \in k[A]$.
\end{theorem}

\begin{proof}
$\Jac(A,w)=0$, so by Jacobi's theorem, Theorem \ref{jacobi} $(2)$,
we obtain that $A$ and $w$ are algebraically dependent over $k$.

By Formanek's observation, Proposition \ref{lemma formanek},
applied to $A$ and $w$
we obtain that there exist $h \in k[x,y]$
and $u(t), v(t) \in k[t]$,
such that
$A=u(h)$ and $w=v(h)$.

Then,
$u'(h)\Jac(h,B)=\Jac(u(h),B)=\Jac(A,B) \in k^*$,
where the first equality can be easily proved.
Therefore, $u'(h) \in k^*$ (and $\Jac(h,B) \in k^*$).

Write
$u(t)=c_mt^m+c_{m-1}t^{m-1}+\cdots+c_2t^2+c_1t+c_0$, $c_j \in k$,
so
$u'(t)=m c_m t^{m-1}+(m-1) c_{m-1} t^{m-2}+\cdots+2 c_2 t+c_1$,
and then
$u'(h)=m c_m h^{m-1}+(m-1) c_{m-1} h^{m-2}+\cdots+2 c_2 h+c_1$.

{}From $u'(h) \in k^*$, we obtain that
$m c_m h^{m-1}+(m-1) c_{m-1} h^{m-2}+\cdots+2 c_2 h+c_1 \in k^*$,
which implies that
$m c_m = (m-1) c_{m-1} = \ldots = 2 c_2 =0$,
$c_1 \in k^*$, and $c_0 \in k$.
Since $k$ is of characteristic zero, we necessarily have
$c_m = c_{m-1} = \ldots = c_2 =0$.

Therefore,
$u(t)=c_1t+c_0$, $c_1 \in k^*$, and $c_0 \in k$,
so
$A=u(h)=c_1 h+c_0$.

Then, $h=\frac{A-c_0}{c_1}$,
which shows that
$w=v(h)=v(\frac{A-c_0}{c_1}) \in k[A]$.
\end{proof}

The above proof of Theorem \ref{cmw char zero} (which is based on Formanek's observation) 
is independent of ~\cite[Theorem 1]{wang younger}, so it can serve as another proof for ~\cite[Theorem 1]{wang younger}.

Our proof of Theorem \ref{cmw char zero} is not valid if we replace $k$ by a field of large enough characteristic
(= a characteristic for which Jacobi's theorem, Theorem \ref{jacobi} $(2)$, is valid);
indeed, although Formanek's observation is valid for any field, we needed zero characteristic for
'$j c_j=0$ implies $c_j=0$, $2 \leq j \leq m$', and we do not have control on $m$.

\subsection{Second way: Dependent on the original CMW's theorem}
We can prove CMW's theorem over $k$, Theorem \ref{cmw char zero},
in a way which is not based on Formanek's observation, but it relies on 
the original proof ~\cite[Theorem 1]{wang younger}.
For convenience, we will divide our proof to three steps:
\begin{itemize}
\item Step 1: Replace $\mathbb{C}$ by an algebraically closed field of characteristic zero.
\item Step 2: Prove the result for any sub-field of an algebraically closed field of characteristic zero.
\item Step 3: Consider $k$ as a sub-field of $\bar{k}$, where $\bar{k}$ is an algebraic closure of $k$.
\end{itemize}

\begin{proposition}[Step 1]\label{step 1}
Let $L$ be an algebraically closed field of characteristic zero.
Suppose $\Jac(A,B) \in L^*$ and $\Jac(A,w)=0$
for $A,B,w \in L[x,y]$.
Then $w \in L[A]$.
\end{proposition}

\begin{proof}
The proof of the original CMW's theorem ~\cite[Theorem 1]{wang younger}
is still valid if we replace $\mathbb{C}$ by $L$.
Indeed, that proof uses three results: Jacobi's theorem, Theorem \ref{jacobi} $(1)$, 
~\cite[Theorem 1]{wang deri} and ~\cite[Corollary 1.5, p. 74]{bass}.
In order to replace $\mathbb{C}$ by any algebraically closed field
of characteristic zero, it is enough to check that each of these three results are valid
over an algebraically closed field of characteristic zero; indeed:
\begin{itemize}
\item Theorem \ref{jacobi} $(1)$ is over any field.
\item ~\cite[Theorem 1]{wang deri} is over any ring.
\item ~\cite[Corollary 1.5, p. 74]{bass} is over an algebraically closed field 
of characteristic zero.
\end{itemize}
\end{proof}

\begin{proposition}[Step 2]\label{step 2}
Let $L$ be an algebraically closed field of characteristic zero,
and let $K \subseteq L$ be any sub-field of $L$.

Suppose $\Jac(A,B) \in K^*$ and $\Jac(A,w)=0$
for $A,B,w \in K[x,y]$.
Then $w \in K[A]$.
\end{proposition}

\begin{proof}
$w \in K[x,y]$, so we can write
$w= \sum_{j=0}^{s}w_j y^j$, for some $w_j \in K[x]$, $0 \leq j \leq s$.
$A \in K[x,y]$, so we can write
$A= \sum_{l=0}^{t}A_l y^l$, for some $A_l \in K[x]$, $0 \leq l \leq t$.

$A,B,w \in K[x,y] \subseteq L[x,y]$ satisfy
$\Jac(A,B) \in K^* \subseteq L^*$
and $\Jac(A,w)=0$,
so Proposition \ref{step 1} implies that $w \in L[A]$,
namely,
$w= \sum_{i=0}^{r} c_i A^i$ for some $c_i \in L$, $0 \leq i \leq r$.

We will show that for every $0 \leq i \leq r$,
$c_i \in K$, by induction on $r$:

If $r=0$, then $w= c_0$.
Then
$\sum_{j=0}^{s}w_j y^j= c_0$,
so
$\sum_{j=0}^{s}w_j y^j- c_0 =0$,
namely,
$\sum_{j=1}^{s}w_j y^j +(w_0 - c_0) =0$.
As an equation in $L[x,y]=L[x][y]$,
we get $w_0 - c_0 = 0$
and $w_j = 0$, $1 \leq j \leq s$.
Hence $K[x] \ni w_0 = c_0 \in L$,
so $c_0 \in L \cap K[x] = K$.

If $r=1$, then $w=c_0 +c_1 A$.
On the one hand,
$w=c_0 +c_1 A= c_0 + c_1(\sum_{l=0}^{t}A_l y^l)
= c_0 + c_1(A_0+A_1y+\cdots+A_{t-1} y^{t-1} + A_t y^t)
= (c_0 + c_1A_0)+ c_1A_1y+ \cdots+ c_1A_{t-1} y^{t-1}+ c_1A_t y^t$;
as an element of $L[x][y]$ its $(0,1)$-degree equals $t$
and its $(0,1)$-leading term is $c_1A_t y^t$.
On the other hand,
$w= \sum_{j=0}^{s}w_j y^j= w_0+ w_1y+ \cdots+ w_{s-1}y^{s-1}+ w_sy^s$;
as an element of $K[x][y] \subseteq L[x][y]$ its $(0,1)$-degree equals $s$
and its $(0,1)$-leading term is $w_s y^s$.
Combining the two yields that $t=s$
and $c_1 A_t = w_s$.
Hence $c_1 = \frac{w_s}{A_t} \in K(x)$,
so $c_1 \in L \cap K(x) = K$.
{}From $w=c_0 +c_1 A$,
we get $c_0 = w- c_1 A \in K[x,y]$.
Hence $c_0 \in L \cap K[x,y] = K$,
so $c_0 \in K$.
Therefore, we obtained that
$w=c_0 +c_1 A \in K[A]$.

If $r \geq 2$:
On the one hand,
$w= \sum_{i=0}^{r} c_i A^i
= \sum_{i=0}^{r} c_i (\sum_{l=0}^{t}A_l y^l)^i
= c_0+c_1(\sum_{l=0}^{t}A_l y^l)+\cdots+c_r(\sum_{l=0}^{t}A_l y^l)^r$;
as an element of $L[x][y]$ its $(0,1)$-degree equals $tr$
and its $(0,1)$-leading term is $c_r A_t^r y^{tr}$.
On the other hand,
$w= \sum_{j=0}^{s}w_j y^j$;
as an element of $K[x][y] \subseteq L[x][y]$ its $(0,1)$-degree equals $s$
and its $(0,1)$-leading term is $w_s y^s$.
Combining the two yields that $tr=s$
and $c_r A_t^r = w_s$.
Hence $c_r = \frac {w_s}{A_t^r} \in K(x)$,
so $c_r \in L \cap K(x) = K$.

Then $\sum_{i=0}^{r-1} c_i A^i = w- c_r A^r \in K[x,y]$.
We can apply the induction hypothesis on
$\sum_{i=0}^{r-1} c_i A^i$ and obtain that
for every $0 \leq i \leq r-1$,
$c_i \in K$.
Therefore, we obtained that 
$w=\sum_{i=0}^{r} c_i A^i \in K[A]$.
\end{proof}

Now we present a second proof for CMW's theorem over $k$,
Theorem \ref{cmw char zero}:
\begin{proposition}[Step 3, second proof for CMW's theorem over $k$]\label{step 3}
Suppose $\Jac(A,B) \in k^*$ and $\Jac(A,w)=0$
for $A,B,w \in k[x,y]$.
Then $w \in k[A]$.
\end{proposition}

\begin{proof}
Let $\bar{k}$ be an algebraic closure of $k$.
We have, $k \subseteq \bar{k}$
with $\bar{k}$ an algebraically closed field
of characteristic zero.
Apply Proposition \ref{step 2} 
and get that $w \in k[A]$.
\end{proof}


\section{The two-dimensional Centralizer Conjecture}

For $u,v \in R[x,y]$, $R$ a commutative ring, 
we will use the following usual terminology:
\begin{itemize}
\item If $\Jac(u,v) \in R[x,y]^*=R^*$, then we say that $u$
has a \textit{Jacobian mate} in $R[x,y]$, $v$ (and vice versa), 
and $u,v$ is a \textit{Jacobian pair} (in $R[x,y]$).
\item If $(x,y) \mapsto (u,v)$ is an automorphism of $R[x,y]$,
then we say that $u,v$ is an \textit{automorphism pair} (in $R[x,y]$).
(An automorphism pair is a Jacobian pair, since the Jacobian
of an automorphism pair is invertible in $R[x,y]$).
\end{itemize}
Also, we will use the following non-usual terminology:
If $\Jac(u,v)=0$, then we say that $v$ is in 'the centralizer 
of $u$ with respect to the Jacobian',
$u$ is in 'the centralizer of $v$ with respect to the Jacobian',
and $u,v$ 'commute with respect to the Jacobian'.

{}From CMW's theorem over $k$, Theorem \ref{cmw char zero}
or Proposition \ref{step 3}, it is immediate to obtain the following:

\begin{theorem}\label{cmw domain}
Suppose $\Jac(A,B) \in D[x,y]^*=D^*$ 
and $\Jac(A,w)=0$ for 

$A,B,w \in D[x,y]$.
Then $w \in Q(D)[A]$.
\end{theorem}

\begin{proof}
$A,B \in D[x,y] \subset Q(D)[x,y]$
satisfy
$\Jac(A,B) \in D^* \subset Q(D)^*$
and
$w \in D[x,y] \subset Q(D)[x,y]$
satisfies $\Jac(A,w)= 0$.
We can apply Theorem \ref{cmw char zero} or Proposition \ref{step 3}
($Q(D)$ is a field of characteristic zero)
and obtain that $w \in Q(D)[A]$.
\end{proof}

Theorem \ref{cmw domain} shows that
there exists $d(w) \in D$ such that $d(w)w \in D[A]$;
indeed, since $w \in Q(D)[A]$ we can write
$w= \sum \frac{\lambda_i}{\mu_i}A^i$,
for some $0 \neq \mu_i, \lambda_i \in D$.
Then,
$w= \frac{1}{d} \sum d_i A^i$,
where
$d= \prod{\mu_i} \in D^*$
and
$d_j= \frac{\lambda_j \prod{\mu_i}}{\mu_j} \in D$.
Therefore,
$dw= d(\frac{1}{d} \sum d_i A^i)
= d\frac{1}{d} (\sum d_i A^i)
= \sum d_i A^i \in D[A]$.
This does not tell much, since each $w$ has its own $d=d(w)$,
so in order to obtain the centralizer of $A$ with respect to the Jacobian,
we will need to add (apriori) infinitely many $\frac{1}{d(w)}$'s to $D$.
In other words, we can say that the centralizer of $A$ is 
$D[\{\frac{1}{d(w)}\}_{w}][A] \subseteq Q(D)[A]$.
We conjecture that the centralizer of $A$ is the smallest possible
(= every $\frac{1}{d(w)}$ belongs to $D$), namely $D[A]$.

\begin{conjecture}[The two-dimensional Centralizer Conjecture over $D$]\label{cent conj int dom}
Suppose $\Jac(A,B) \in D[x,y]^*=D^*$ and $\Jac(A,w)=0$
for $A,B,w \in D[x,y]$.
Then $w \in D[A]$.
\end{conjecture}

\begin{examples}[Non-counterexamples]\label{examples cmw}

\textbf{First non-counterexample:}
The following is not a counterexample to 
the two-dimensional Centralizer Conjecture over $\mathbb{Z}$:
$A=2+3y$, $w=1+y$; clearly $w$ is in the centralizer of $A$ w.r.t. the Jacobian,
and $w=1+y=\frac 13+\frac 13 (2+3y)=\frac 13+\frac 13 A 
\in \mathbb{Q}[A]-\mathbb{Z}[A]$.
However, $A$ does not have a Jacobian mate in $\mathbb{Z}[x,y]$,
since if $B \in \mathbb{Z}[x,y]$ was a Jacobian mate of $A$,
then, by definition,
$-3B_x= \Jac(A,B) \in \mathbb{Z}[x,y]^*=\mathbb{Z}^*= \{1,-1\}$
($B_x$ denotes the derivative of $B$ w.r.t. $x$),
so $B_x = \pm \frac 13 \notin \mathbb{Z}[x,y]$,
which is impossible (since $B \in \mathbb{Z}[x,y]$
implies that $B_x \in \mathbb{Z}[x,y]$).

\textbf{Second non-counterexample:}
The following is not a counterexample to 
the two-dimensional Centralizer Conjecture over 
$D:= k[a^2,ab,b^2,a^3,b^3]$, where $a$ and $b$ are transcendental over $k$;
it is due to YCor ~\cite{my}: 
$A=(ax+by)^2=a^2x^2+2abxy+b^2y^2 \in D[x,y]$, 
$w=(ax+by)^3=a^3x^3+3a^2bx^2y+3ab^2xy^2+b^3y^3$; 
clearly $w$ is in the centralizer of $A$ w.r.t. the Jacobian.
{}From considerations of degrees, $w \notin D[A]$
and also from considerations of degrees, $w \notin Q(D)[A]$
(the degree of $w$ is $3$, while the degree of a polynomial in
$A$ is even).
However, $A$ does not have a Jacobian mate in $D[x,y]$;
this can be seen in two ways:
\begin{itemize}
\item Otherwise, by Theorem \ref{cmw domain},
$w \in Q(D)[A]$, but $w \notin Q(D)[A]$.
\item Otherwise, if $B \in D[x,y]$ was a Jacobian mate of $A$,
then, by definition,
$\Jac(A,B) \in D[x,y]^*=D^*=k^*$
A direct computation shows that
$\frac 12 \Jac(A,B)= (a^2x+aby)B_y-(b^2y+abx)B_x$,
hence $\Jac(A,B)|(x=0,y=0)=0$, 
which contradicts $\Jac(A,B) \in k^*$.
\end{itemize}
\end{examples}

The two-dimensional Centralizer Conjecture has a positive answer (at least)
in the following two special cases:
\begin{itemize}
\item First special case: $D$ is, in addition, a field; this is CMW's theorem over $k$, 
Theorem \ref{cmw char zero} or Proposition \ref{step 3}.
\item Second special case: $(x,y) \mapsto (A,B)$ is, in addition, an automorphism
($A,B$ is an automorphism pair);
this is the content of the following Theorem \ref{cmw domain auto}.
\end{itemize}

\begin{theorem}[Special case: Automorphism pair]\label{cmw domain auto}
Suppose $\Jac(A,B) \in D[x,y]^*=D^*$ and $\Jac(A,w)=0$
for $A,B,w \in D[x,y]$.
Further assume that $(x,y) \mapsto (A,B)$
is an automorphism of $D[x,y]$.
Then $w \in D[A]$.
\end{theorem}

\begin{proof}
Denote
$g: (x,y) \mapsto (A,B)$.
By assumption, $g$ is an automorphism
of $D[x,y]$,
so $g^{-1}: D[x,y] \to D[x,y]$
exists and is also an automorphism of $D[x,y]$.

\textbf{First step $A=x$:}
If $A=x$, then it is easy to see that $w \in D[A]= D[x]$.
Indeed,
$w_y = \Jac(x,w)= \Jac(A,w)= 0$,
where $w_y$ is the derivative of $w$
with respect to $y$.
Write $w$ in its $(0,1)$-degree form:
$w=w_0+w_1y+w_2y^2+\cdots+w_my^m$,
where
$w_0,w_1,w_2,\ldots,w_m \in D[x]$.
We have just seen that $w_y= 0$,
hence,
$0= w_y= w_1+2w_2y+\cdots+mw_my^{m-1}$,
which implies that
$w_1=2w_2= \ldots =mw_m=0$,
so
$w_1=w_2= \ldots =w_m=0$
(since the characteristic of $D$ is zero).
Therefore, $w=w_0 \in D[x]= D[A]$,
and we are done.

\textbf{Second step $A$ arbitrary:}
The general case can be obtained from the special case $A=x$.

Claim: $u:= g^{-1}(w)$ is in the centralizer of $x$.
Proof of Claim:
$\Jac(x,u)=\Jac(g^{-1}(A),g^{-1}(w))=0$,
since $\Jac(A,w)=0$
(it is clear that if $A$ and $w$ are algebraically dependent
over $D$, then $g^{-1}(A)$ and $g^{-1}(w)$
are also algebraically dependent over $D$).
Apply the first step and get that
$g^{-1}(w) = u \in D[x]$.
Hence we can write
$g^{-1}(w)= d_0+d_1x+d_2x^2+\cdots+d_lx^l$,
for some
$d_0,d_1,d_2,\ldots,d_l \in D$.
Then,
$w= g(d_0+d_1x+d_2x^2+\cdots+d_lx^l)
= g(d_0)+g(d_1x)+g(d_2x^2)+\cdots+g(d_lx^l)
= g(d_0)+g(d_1)g(x)+g(d_2)g(x^2)+\cdots+g(d_l)g(x^l)
= d_0+d_1A+d_2A^2+\cdots+d_lA^l \in D[A]$.
\end{proof}

It is not surprising that we have not succeeded to find a counterexample
to the two-dimensional Centralizer Conjecture, 
since when we picked an $A \in D[x,y]$ which has a Jacobian mate $B \in D[x,y]$,
those $A$ and $B$ always defined an automorphism of $D[x,y]$,
and for an automorphism of $D[x,y]$ (= $A$ is part 
of an automorphism pair) the conjecture holds by Theorem \ref{cmw domain auto}.

Given a commutative ring $R$, denote:
\begin{itemize}
\item $JC(2,R)$: The two-dimensional Jacobian Conjecture over $R$ is true, 
namely, every $R$-algebra endomorphism of $R[x,y]$ 
having an invertible Jacobian (the Jacobian is in 
$R[x,y]^*=R^*$) is an automorphism of $R[x,y]$.
\item $CC(2,R)$: The two-dimensional Centralizer Conjecture over $R$ is true,
namely, if $\Jac(A,B) \in R[x,y]^*=R^*$ and $\Jac(A,w)=0$
for $A,B,w \in R[x,y]$, then $w \in R[A]$.
\end{itemize}

We have,
\begin{theorem}[$JC(2,D)$ implies $CC(2,D)$]\label{more general than jc}
If the two-dimensional Jacobian Conjecture over $D$ is true,
then the two-dimensional Centralizer Conjecture over $D$ is true.
\end{theorem}

We do not know if the converse of Theorem \ref{more general than jc} is true, namely, 
if a positive answer to the two-dimensional Centralizer Conjecture over $D$
implies a positive answer to the two-dimensional Jacobian Conjecture over $D$.

\begin{proof}
Assume that $A \in D[x,y]$ has a Jacobian mate
$B$ in $D[x,y]$,
and $w \in D[x,y]$ satisfies $\Jac(A,w)= 0$.
We must show that $w \in D[A]$.
By assumption, the two-dimensional Jacobian Conjecture over $D$ is true,
so $(x,y) \mapsto (A,B)$ is an automorphism of $D[x,y]$.
Apply Theorem \ref{cmw domain auto} and get that $w \in D[A]$.
\end{proof}

Recall ~\cite[Lemma 1.1.14]{essen book} which says that
$JC(2,\mathbb{C})$ implies $JC(2,D)$, therefore:

\begin{theorem}\label{more general than jc 2}
If the two-dimensional Jacobian Conjecture over $\mathbb{C}$ is true,
then the two-dimensional Centralizer Conjecture over $D$ is true.
\end{theorem}

\begin{proof}
$JC(2,\mathbb{C}) \Rightarrow JC(2,D) \Rightarrow CC(2,D)$,
where the first implication is due to ~\cite[Lemma 1.1.14]{essen book},
while the second implication is due to Theorem \ref{more general than jc}.
\end{proof}

Of course, Theorem \ref{more general than jc} implies that
if there exists an integral domain $D$ of characteristic zero
such that the two-dimensional Centralizer Conjecture over $D$ is false,
then the two-dimensional Jacobian Conjecture over $D$ is false,
and then, by ~\cite[Lemma 1.1.14]{essen book}, 
the two-dimensional Jacobian Conjecture over $\mathbb{C}$ is false.

We suspect that the two-dimensional Centralizer Conjecture over \textit{any} integral domain
of characteristic zero is true, because we 'believe' that the two-dimensional
Jacobian Conjecture is true.

Finally, we wish to quote A. van den Essen ~\cite[page 2]{essen amazing image}: 
``... All these experiences fed my believe that the Jacobian Conjecture, if
true at all, would be difficult to generalize, since it felt like a kind of optimal
statement".

\section{Non-commutative analog: The first Weyl algebra}

By definition, the $n$'th Weyl algebra $A_n(k)$ is the unital associative $k$-algebra 
generated by $2n$ elements $X_1, \ldots, X_n, Y_1, \ldots, Y_n$ 
subject to the following defining relations: 
$[Y_i,X_j]= \delta_{ij}$, $[X_i,X_j]= 0$ and $[Y_i,Y_j]= 0$, 
where $\delta_{ij}$ is the Kronecker delta. 
When $n=1$, we will denote the generators by $X,Y$ 
instead of by $X_1,Y_1$.

The first Weyl algebra, $A_1(k)$, was first studied by Dirac ~\cite{dirac}.
In ~\cite{dixmier}, Dixmier posed six questions concerning $A_1(k)$,
$k$ is a field of characteristic zero. The first question asked if every 
$k$-algebra endomorphism of $A_1(k)$ is an automorphism of $A_1(k)$;
this is known as Dixmier Conjecture.

Given a commutative ring $R$, denote:
\begin{itemize}
\item $JC(n,R)$: The $n$-dimensional Jacobian Conjecture over $R$ is true, 
namely, every $R$-algebra endomorphism of $R[x_1,\ldots,x_n]$ 
having an invertible Jacobian (the Jacobian is in 
$R[x_1,\ldots,x_n]^*=R^*$) is an automorphism of $R[x_1,\ldots,x_n]$.
\item $DC(n,R)$: The $n$'th Dixmier Conjecture over $R$ is true, namely,
every $R$-algebra endomorphism of $A_n(R)$ is an automorphism of $A_n(R)$.
\item $JC(\infty,R)$: For all $n \in \mathbb{N}$, 
the $n$-dimensional Jacobian Conjecture over $R$ is true.
\item $DC(\infty,R)$: For all $n \in \mathbb{N}$, 
the $n$'th Dixmier Conjecture over $R$ is true.
\end{itemize}

There is a well-known connection between the Jacobian Conjecture and the Dixmier Conjecture,
which says the following:
For all $n \in \mathbb{N}$:
\begin{itemize} 
\item [(i)] $DC(n,k)$ $\Rightarrow$ $JC(n,k)$.
\item [(ii)] $JC(2n,k)$ $\Rightarrow$ $DC(n,k)$.
\end{itemize}
The first result can be found in ~\cite[Theorem 4.2.8]{essen book}
and in ~\cite[page 297]{bcw} (immediately after Proposition 2.3).
The second result was proved independently by Y. Tsuchimoto ~\cite{tsuchimoto} 
and by A. Belov-Kanel and M. Kontsevich ~\cite{belov}.
There exist additional proofs due to V. V. Bavula ~\cite{bavula jacobian},
and due to K. Adjamagbo and A. van den Essen ~\cite{essen adja}.
{}From $(i)$ and $(ii)$ it is clear that $JC(\infty,k)$ is 
equivalent to $DC(\infty,k)$.

In the non-commutative algebra $A_1(k)$,
denote the centralizer of an element $P$ by $\Cent(P)$
(= all the elements in $A_1(k)$ which commute with $P$).

An analog result in $A_1(k)$ to CMW's result in $k[x,y]$ exists,
and is due to J. A. Guccione, J. J. Guccione and C. Valqui, namely:

\begin{theorem}[Analog of CMW's theorem over $k$]\label{analog cmw}
If $P,Q \in A_1(k)$ satisfy $[Q,P] \in A_1(k)^*=k^*$, 
then $\Cent(P)=k[P]$.
\end{theorem}

\begin{proof}
See ~\cite[Theorem 2.11]{ggv cent}.
\end{proof}

Notice that the non-commutative case ~\cite[Theorem 2.11]{ggv cent} 
is already over $k$, while the commutative case ~\cite[Theorem 1]{wang younger} 
is over $\mathbb{C}$ and we showed in Theorem \ref{cmw char zero} and in 
Proposition \ref{step 3} that $\mathbb{C}$ can be replaced by $k$.

\begin{remark}
It seems that none of the proofs of Theorem \ref{cmw char zero}
and Proposition \ref{step 3} has a non-commutative version:
\begin{itemize}
\item The proof of Theorem \ref{cmw char zero} is based on 
Formanek's observation, Proposition \ref{lemma formanek},
and on Jacobi's theorem, Theorem \ref{jacobi} $(2)$, which together
say the following:
Assume that $C,D \in k[x,y]$ satisfy $\Jac(C,D)=0$.
Then there exist $h \in k[x,y]$
and $u(t), v(t) \in k[t]$,
such that
$C=u(h)$ and $D=v(h)$.
However, there exists a known counterexample due to Dixmier ~\cite{dixmier}
to this result in the non-commutative case, namely,
there exist $C,D \in A_1(k)$ satisfying $[C,D]=0$,
but there exist no $h \in A_1(k)$
and $u(t), v(t) \in k[t]$,
such that
$C=u(h)$ and $D=v(h)$.
For more details on the counterexample, see ~\cite{makar cent weyl}
and ~\cite{mo makar}.
\item The proof of Proposition \ref{step 3} is based on the original CMW's 
theorem; the proof of the original CMW's theorem uses results concerning 
$\mathbb{C}[x,y]$ (such as extension of derivations), and we do not know if
those results have non-commutative analogs.
\end{itemize} 
\end{remark}

CMW's theorem over $k$ implies Theorem \ref{cmw domain},
and similarly, the analog of CMW's theorem over $k$, Theorem \ref{analog cmw}, 
implies the following analog of Theorem \ref{cmw domain}:

\begin{theorem}\label{cmw domain dix}
Suppose $[A,B] \in A_1(D)^*=D^*$ and $[A,w]=0$
for $A,B,w \in A_1(D)$.
Then $w \in Q(D)[A]$.
\end{theorem}

\begin{proof}
$A,B \in A_1(D) \subset A_1(Q(D))$
satisfy
$[A,B] \in D^* \subset Q(D)^*$
and
$w \in A_1(D) \subset A_1(Q(D))$
satisfies $[A,w]= 0$.
We can apply Theorem \ref{analog cmw}
($Q(D)$ is a field of characteristic zero)
and obtain that $w \in Q(D)[A]$.
\end{proof}

We conjecture the non-commutative analog of Conjecture \ref{cent conj int dom}:
\begin{conjecture}[The (first) Centralizer Conjecture over $D$]\label{cent conj int dom dix}
Suppose $[A,B] \in A_1(D)^*=D^*$ and $[A,w]=0$
for $A,B,w \in A_1(D)$.
Then $w \in D[A]$.
\end{conjecture}

The following useful equation will be applied in Examples \ref{examples cmw dix}
and in the proof of Theorem \ref{cmw domain auto dix}.

\begin{lemma}[Useful equation]\label{useful}
For all 
$t(Y) \in k[Y]$ and $i \in \mathbb{N}$:
$[t(Y),X^i]=i X^{i-1}t'(Y)+\binom{i}{2} X^{i-2}t''(Y)+\dots$.
\end{lemma}

In particular, this equation is valid for $t(Y) \in D[Y]$
(since $t(Y) \in D[Y] \subset Q(D)[Y]$).

\begin{proof}
See ~\cite[Proposition 1.6]{ggv dix}.
\end{proof}

The analog of Examples \ref{examples cmw} is:

\begin{examples}[Non-counterexamples]\label{examples cmw dix}

\textbf{First non-counterexample:}
The following is not a counterexample to 
the (first) Centralizer Conjecture over $\mathbb{Z}$:
$A=2+3Y$, $w=1+Y$; clearly $w \in \Cent(A)$,
and $w=1+Y=\frac 13+\frac 13 (2+3Y)=\frac 13+\frac 13 A 
\in \mathbb{Q}[A]-\mathbb{Z}[A]$.
However, $A$ does not have a Dixmier mate in $A_1(\mathbb{Z})$,
since if $B \in A_1(\mathbb{Z})$ was a Dixmier mate of $A$,
then by the definition of a Dixmier mate,
$[A,B] \in A_1(\mathbb{Z})^*=\mathbb{Z}^*= \{1,-1\}$.
But this is impossible from the following direct computation:
Write $B=\sum b_{ij}X^iY^j$, $b_{ij} \in \mathbb{Z}$.
Then,

$[B,A]= [\sum b_{ij}X^iY^j,2+3Y]= \sum b_{ij}[X^iY^j,2+3Y]$
$= \sum b_{ij}(X^i[Y^j,2+3Y]+[X^i,2+3Y]Y^j)$
$=\sum b_{ij}[X^i,2+3Y]Y^j= \sum 3b_{ij}[X^i,Y]Y^j$ 
$=-\sum 3b_{ij}[Y,X^i]Y^j= -\sum 3b_{ij}iX^{i-1}Y^j$ 
$=-\sum 3ib_{ij}X^{i-1}Y^j$.
If $[A,B] \in \{1,-1\}$,
then 
$\sum 3ib_{ij}X^{i-1}Y^j \in \{1,-1\}$.
Therefore, $3b_{10} \in \{1,-1\}$,
so $b_{10} \in \{\frac 13,-\frac 13\}$,
a contradiction to $b_{10} \in \mathbb{Z}$.

\textbf{Second non-counterexample:}
The following is not a counterexample to 
the (first) Centralizer Conjecture over 
$D:= k[a^2,ab,b^2,a^3,b^3]$, where $a$ and $b$ are transcendental over $k$:
$A=(aX+bY)^2$, $w=(aX+bY)^3$.
$w \in \Cent(A)$ (clear), $w \notin Q(D)[A]$ (considerations of degrees),
$A$ does not have a Dixmier mate in $A_1(\mathbb{D})$
(direct computation).
\end{examples}

The (first) Centralizer Conjecture has a positive answer (at least)
in the following two special cases:
\begin{itemize}
\item First special case: $D$ is, in addition, a field; 
this is the analog of CMW's theorem over $k$, 
Theorem \ref{analog cmw}.
\item Second special case: $(X,Y) \mapsto (A,B)$ is, in addition, an automorphism
($A,B$ is an automorphism pair);
this is the content of the following Theorem \ref{cmw domain auto dix}.
\end{itemize}

\begin{theorem}[Special case: Automorphism pair]\label{cmw domain auto dix}
Suppose $[A,B] \in A_1(D)^*=D^*$ and $[A,w]=0$
for $A,B,w \in A_1(D)$.
Further assume that $(X,Y) \mapsto (A,\lambda B)$,
$\lambda=[A,B]^{-1}$, is an automorphism of $A_1(D)$.
Then $w \in D[A]$.
\end{theorem}

\begin{proof}
Denote
$g: (X,Y) \mapsto (A,\lambda B)$.
By assumption, $g$ is an automorphism
of $A_1(D)$,
so $g^{-1}: A_1(D) \to A_1(D)$
exists and is also an automorphism of $A_1(D)$.

Write $w=\sum w_{ij}X^iY^j$, $w_{ij} \in D$,
$0 \leq i \leq m$, $0 \leq j \leq l$.

\textbf{First step $A=X$:}
If $A=X$, then it is easy to see that $w \in D[A]= D[X]$.
Indeed,
$\sum jw_{ij}X^iY^{j-1}=\sum w_{ij}X^i jY^{j-1} 
=\sum w_{ij}X^i[Y^j,X]=\sum w_{ij}(X^i[Y^j,X]+[X^i,X]Y^j) 
=\sum w_{ij}[X^iY^j,X]=[\sum w_{ij}X^iY^j,X] = [w,X]= [w,A]= 0$.

Therefore, $jw_{ij}=0$, 
$0 \leq i \leq m$, $1 \leq j \leq l$.
Since $D$ is of characteristic zero, 
$w_{ij}=0$, 
$0 \leq i \leq m$, $1 \leq j \leq l$.
Then,
$w=\sum_{i=0}^{m} w_{i0}X^iY^0 
= \sum_{i=0}^{m} w_{i0}X^i \in D[X]$.

\textbf{Second step $A$ arbitrary:}
The general case can be obtained from the special case $A=X$.

Claim: $u:= g^{-1}(w)$ is in the centralizer of $X$.

Proof of Claim: 
$[X,u]=[g^{-1}(A),g^{-1}(w)]$
$=g^{-1}(A)g^{-1}(w)-g^{-1}(w)g^{-1}(A)$
$=g^{-1}(Aw)-g^{-1}(wA)=g^{-1}(Aw-wA)$
$=g^{-1}([A,w])=g^{-1}(0)=0$.
Apply the first step and get that
$g^{-1}(w) = u \in D[X]$.
Hence we can write
$g^{-1}(w)= d_0+d_1X+d_2X^2+\cdots+d_lX^l$,
for some
$d_0,d_1,d_2,\ldots,d_l \in D$.
Then,
$w= g(d_0+d_1X+d_2X^2+\cdots+d_lX^l)
= g(d_0)+g(d_1X)+g(d_2X^2)+\cdots+g(d_lX^l)
= g(d_0)+g(d_1)g(X)+g(d_2)g(X^2)+\cdots+g(d_l)g(X^l)
= d_0+d_1A+d_2A^2+\cdots+d_lA^l \in D[A]$.
\end{proof}

We have an analog to Theorem \ref{more general than jc}:
\begin{theorem}\label{more general than jc dix}
If the first Dixmier Conjecture over $D$ is true,
then the (first) Centralizer Conjecture over $D$ is true.
\end{theorem}

We do not know if the converse of Theorem \ref{more general than jc dix} is true, namely, 
if a positive answer to the (first) Centralizer Conjecture over $D$
implies a positive answer to the first Dixmier Conjecture over $D$.

\begin{proof}
Assume that $A \in A_1(D)$ has a Dixmier mate
$B$ in $A_1(D)$,
and $w \in A_1(D)$ satisfies $[A,w]= 0$.
We must show that $w \in D[A]$.
By assumption, the first Dixmier Conjecture over $D$ is true,
so $(X,Y) \mapsto (A,\lambda B)$, $\lambda=[A,B]^{-1}$,
is an automorphism of $A_1(D)$.
Apply Theorem \ref{cmw domain auto dix} and get that $w \in D[A]$.
\end{proof}

We do not know if an analog of ~\cite[Lemma 1.1.14]{essen book} exists,
hence we do not know how to obtain an analog of Theorem \ref{more general than jc 2}.

Of course, Theorem \ref{more general than jc dix} implies that
if there exists an integral domain $D$ of characteristic zero
such that the first Centralizer Conjecture over $D$ is false,
then the first Dixmier Conjecture over $D$ is false,
namely, there exists $A,B \in A_1(D)$,
$[B,A]=1$ and $(X,Y) \mapsto (A,B)$ is not an automorphism of 
$A_1(D)$. We have recalled above in $(ii)$ that $JC(2n,k)$ $\Rightarrow$ $DC(n,k)$.
If this implication is also valid if we replace $k$ by $D$,
namely, $JC(2n,D)$ $\Rightarrow$ $DC(n,D)$,
then, in particular, $JC(2,D)$ $\Rightarrow$ $DC(1,D)$.
Therefore, if the first Dixmier Conjecture over $D$ is false,
then the two-dimensional Jacobian Conjecture over $D$ is false,
and then by ~\cite[Lemma 1.1.14]{essen book},
the two-dimensional Jacobian Conjecture over $\mathbb{C}$ is false.


\bibliographystyle{plain}

\end{document}